\documentclass[a4paper,12pt]{article}

\usepackage{amsmath}
\usepackage{amsthm}
\usepackage{amssymb}

\usepackage{geometry}

\usepackage[english]{babel}

\newcommand{\tb}{\boldsymbol}
\newcommand{\bo}{\bold}
\newcommand{\mbb}{\mathbb}

\newcommand{\leb}{\lambda}
\newcommand{\ld}{\ldots}
\newcommand{\ra}{\rightarrow}
\newcommand{\D}{\Delta_{K_h}}

\newcommand{\bsa}{\boldsymbol{a}}
\newcommand{\bsb}{\boldsymbol{b}}
\newcommand{\bsf}{\boldsymbol{f}}
\newcommand{\bst}{\boldsymbol{t}}

\newcommand{\bsv}{\boldsymbol{v}}
\newcommand{\bsw}{\boldsymbol{w}}
\newcommand{\bsx}{\boldsymbol{x}}
\newcommand{\bsy}{\boldsymbol{y}}

\newcommand{\bsalpha}{\boldsymbol{\alpha}}
\newcommand{\bsbeta}{\boldsymbol{\beta}}

\newcommand{\bszero}{\boldsymbol{0}}
\newcommand{\bsone}{\boldsymbol{1}}

\newcommand{\cP}{\mathcal{P}}

\newcommand{\lc}{\left\lceil}
\newcommand{\rc}{\right\rceil}
\newcommand{\lo}{\lceil \log_qd \rceil}
\newcommand{\be}{\boldsymbol{\beta}}

\newcommand{\Zb}{\overline{\mbb{Z}}_q((t^{-1}))}
\newcommand{\Z}{\mbb{Z}_q((t^{-1}))}
\newcommand{\Zm}{\mbb{Z}_q^{d \times (n-1)}}
\newcommand{\ZZ}{\mathbb{Z}}

\theoremstyle{plain}
\newtheorem{thm}{Theorem}
\newtheorem{defi}{Definition}
\newtheorem{lem}[thm]{Lemma}
\newtheorem{cor}[thm]{Corollary}
\newtheorem{claim}{Claim}

\newtheoremstyle{myRemark}
	{\topsep}   
  {\topsep}   
  {\normalfont}  
  {0pt}       
  {\itshape} 
  {.}         
  {5pt plus 1pt minus 1pt} 
  {}          
	
\theoremstyle{myRemark}

\newtheorem{rem}{Remark}

\allowdisplaybreaks

\begin{document}

\title{Metrical star discrepancy bounds for lacunary subsequences of digital Kronecker-sequences and polynomial tractability}
\author{Mario Neum\"uller and Friedrich Pillichshammer\thanks{The authors are supported by the Austrian Science Fund (FWF): Projects F5505-N26 (M.N.) and F5509-N26 (F.P.), which are parts of the Special Research Program ``Quasi-Monte Carlo Methods: Theory and Applications''.}}

\date{}

\maketitle

\begin{abstract}
The star discrepancy $D_N^*(\cP)$ is a quantitative measure for the irregularity of distribution of a finite point set $\cP$ in the multi-dimensional unit cube which is intimately related to the integration error of quasi-Monte Carlo algorithms. It is known that for every integer $N \ge 2$ there are point sets $\cP$ in $[0,1)^d$ with $|\cP|=N$ and $D_N^*(\cP) =O((\log N)^{d-1}/N)$. However, for small $N$ compared to the dimension $d$ this asymptotically excellent bound is useless (e.g. for $N \le {\rm e}^{d-1}$). 

In 2001 it has been shown by Heinrich, Novak, Wasilkowski and Wo\'{z}niakowski that for every integer $N \ge 2$ there exist point sets $\cP$ in $[0,1)^d$ with $|\cP|=N$ and $D_N^*(\cP) \le C \sqrt{d/N}$. Although not optimal in an asymptotic sense in $N$, this upper bound has a much better (and even optimal) dependence on the dimension $d$. 

Unfortunately the result by Heinrich et al. and also later variants thereof by other authors are pure existence results and until now no explicit construction of point sets with the above properties is known. Quite recently L\"obbe studied lacunary subsequences of Kronecker's $(n \bsalpha)$-sequence and showed a metrical discrepancy bound of the form $C \sqrt{d (\log d)/N}$ with implied absolute constant $C>0$ independent of $N$ and $d$. 

In this paper we show a corresponding result for digital Kronecker sequences, which are a non-archimedean analog of classical Kronecker sequences.
\end{abstract}

\centerline{\begin{minipage}[hc]{130mm}{
}
\end{minipage}} 

\section{Introduction}

For an $N$-element point set $\cP=\{\bsx_1,\ldots,\bsx_N\}$ in the $d$-dimensional unit cube $[0,1)^d$ the {\it star discrepancy} $D_N^*$ is defined as $$D_N^*(\cP)=\sup_J \left|\frac{A(J,\cP)}{N}-\lambda(J)\right|,$$ where the supremum is extended over all intervals of the form $J=[\bszero,\bst)=\prod_{j=1}^d[0,t_j)$ with $t_j \in [0,1]$, $\bst=(t_1,\ldots,t_d)$, $A(J,\cP)$ is the number of indices $n \in \{1,2,\ldots,N\}$ for which $\bsx_n$ belongs to $J$ and $\lambda(J)$ is the Lebesgue measure of $J$. 

The star discrepancy is a quantitative measure for irregularity of distribution of a point set $\cP$. It is also intimately related to the integration error of a quasi-Monte Carlo (QMC) algorithm via the celebrated Koksma-Hlawka inequality. More information about star discrepancy and its relation to uniform distribution theory and numerical integration can be found in the books \cite{DP10,DT,KN,LP14a,N92}.

Let ${\rm disc}^*(N,d)=\inf_{\cP} D_N^*(\cP)$, where the infimum is extended over all $N$-element point sets $\cP$ in $[0,1)^d$, be the {\it minimal star discrepancy} and let, for $\varepsilon \in (0,1]$, the {\it inverse of the star discrepancy} be defined as $$N^*(\varepsilon,d)=\min\{N \in \mathbb{N}\ : \ {\rm disc}^*(N,d) \le \varepsilon\}.$$

For fixed dimension $d\ge 2$  it is known that there exist $0 < c_d < C_d$ and $\eta_d\in (0,\tfrac{1}{2})$ such that $$c_d \frac{(\log N)^{\frac{d-1}{2}+\eta_d}}{N} \le {\rm disc}^*(N,d) \le C_d \frac{(\log N)^{d-1}}{N}  \ \ \ \mbox{ for all $N \ge 2$.} $$ The lower bound was shown by Bilyk, Lacey and Vagharshakyan~\cite{BLV08} and for the upper bound there are several explicit constructions 
of point sets which achieve such a star discrepancy (see, for example, \cite{DP10,N92} and the references therein). For growing $d$ the exponent $\eta_d$ in the lower bound tends to zero (approximately with order $d^{-2}$). It should be mentioned that the exact determination of the power in the $\log N$-term of the minimal star discrepancy is a very famous and difficult open problem.

In this paper we consider a different view point. It was pointed out in several discussions that the excellent asymptotic behavior of the minimal star discrepancy of $N$-element point sets is not very useful for practical applications, especially when the dimension $d$ is not small. For example it should be noted that $N \mapsto (\log N)^{d-1}/N$ does not start to decrease until $N \ge \exp(d-1)$ and this number is huge already for moderately large $d$. In applications of QMC-algorithms however the dimension $d$ could be in the hundreds (see \cite{DKS,LP14a}).    

Since the last one and a half decades a lot of effort has been put into the analysis of the star discrepancy with respect to dimensions $d$ tending to infinity. In a seminal work by Heinrich, Novak, Wasilkowski and Wo\'{z}niakowski \cite{hnww} it has been shown that there exists an absolute constant $C>0$ such that 
\begin{equation}\label{dihnww}
{\rm disc}^*(N,d) \le C \sqrt{\frac{d}{N}} \ \ \ \mbox{ for all $d,N \in \mathbb{N}$}
\end{equation}
(see \cite[Theorem~3]{hnww}). Later Aistleitner \cite[Theorem~1]{aist11} showed that the constant $C$ can be chosen as $C=10$. From this we obtain $$N^*(\varepsilon,d) \le 100 d \varepsilon^{-2}.$$ In the language of ``Information Based Complexity'' one says that the star discrepancy is {\it polynomially tractable}, see \cite{NW08,NW10}.

On the other hand Hinrichs \cite{h2004} showed that $N^*(\varepsilon,d) \ge c d \varepsilon^{-1}$ for all $d \in \mathbb{N}$ and for sufficiently small $\varepsilon>0$. Hence {\it the inverse of the star discrepancy depends linearly on the dimension}, which is also the programmatic title of \cite{hnww}.

The proof in \cite{aist11} and also the proof of the slightly weaker bound \begin{equation}\label{dihnww1}
{\rm disc}^*(N,s) \le C \sqrt{\frac{d}{N}} \ (\log d + \log N)^{1/2}
\end{equation}                                                                                                                                                                                                      in \cite[Theorem~1]{hnww} (which still implies that $N^*(\varepsilon,d) \le C d \varepsilon^{-2} \log (d/\varepsilon) $) use the probabilistic method. The main ingredient is the fact that one can obtain extremely small probabilities for the deviation from the mean for sums of independent random variables. This probability can be quantified with the help of Bernstein's (in \cite{aist11}) or Hoeffding's (in \cite{hnww}) inequality, respectively. In fact, the point sets in \cite{aist11,hnww} consist of $N$ independently chosen random points from the unit cube $[0,1)^d$. So far no explicit construction of point sets whose star discrepancy satisfies a bound like \eqref{dihnww} or \eqref{dihnww1} is known. 

Some authors, initiated in \cite{DGS}, presented algorithmic constructions of point sets with star discrepancy of order \eqref{dihnww1}. We refer to the survey \cite{gnew12} for more information and references in this direction. However, all these constructions have the disadvantage that their run times are too large in order to be applied in practical applications with large dimension $d$. So there is still need for a really explicit construction.

In 2014 L\"obbe \cite{loeb} studied lacunary subsequences of Kronecker-sequences $(\{n \bsalpha\})_{n \ge 0}$, where $\bsalpha \in \mathbb{R}^d$ and where $\{\cdot\}$ denotes the fractional part applied component-wise to a vector (until now the paper is only available via arXiv.org). Based on the work of Aistleitner, L\"obbe was able to prove the following remarkable metrical result which can be interpreted as a semi-probabilistic (or semi-constructive) version of \eqref{dihnww1}.  

For $\bsalpha \in [0,1)^d$ let $\cP_N(\bsalpha)=\{\bsx_1,\ldots,\bsx_N\}$ be the point set consisting of the first $N$ elements of the infinite sequence $(\bsx_n)_{n \ge 1}$ in $[0,1)^d$ with $\bsx_n=\{ 2^{n-1} \bsalpha\}$ for $n \in \mathbb{N}$.

\begin{thm}[L\"obbe {\cite[Theorem~1.1]{loeb}}]\label{thm0}
Let $N \ge 1$ and $d \ge 2$ be integers. Then for every $\varepsilon \in (0,1)$ there is a quantity $C(\varepsilon)>0$ such that  the star discrepancy of the point set $\cP_N(\bsalpha)$ satisfies $$D_N^*(\cP_N(\bsalpha)) \le C(\varepsilon) \sqrt{\frac{d \log d}{N}}$$ with probability at least $1-\varepsilon$. The quantity $C(\varepsilon)$ is of order $C(\varepsilon) \asymp \log \varepsilon^{-1}$.
\end{thm}

The main problem in the proof of this result is to prove independence of certain random variables in order to be able to apply Bernstein's inequality. Of course, the elements of the classical Kronecker-sequence are not independent. For this reason the author studied lacunary subsequences of the form $(\{2^{n-1} \bsalpha\})_{n \ge 1}$ which led then to the desired independence properties. 

Theorem~\ref{thm0} makes an assertion for fixed $N$, i.e. for finite point sets. In 2007 Dick \cite{D07} considered the problem of the dependence of star discrepancy on the dimension $d$ also for infinite sequences and he gave an existence result. Compared to the bound \eqref{dihnww} for finite point sets the generalization is penalized with an extra $\sqrt{\log N}$-factor in the discrepancy estimate. Later Aistleitner \cite{aist13} improved this further and got rid of the $\sqrt{\log N}$-term. In contrast to the probabilistic approaches in, e.g., \cite{hnww,D07}, the proof in  \cite{aist13} is, like in \cite{loeb}, also of a semi-probabilistic nature in the sense that certain coordinates of the points are deterministic others are chosen randomly. This once more shows the relevance of semi-probabilistic construction in this context.

The following corollary to Theorem~\ref{thm0} addresses a metrical result for infinite sequences:

\begin{cor}\label{cor0}
Let $d\in \mbb{N}$ with $d\geq 2$. Then for every $\delta \in (0,1)$ there is a quantity $C(\delta)>0$ such that the star discrepancy of  $\cP_N(\bsalpha)$ satisfies 
	\begin{align*}
	D^*_N(\cP_N(\bsalpha)) \leq C(\delta) (\log N)\sqrt{\frac{d\log d}{N}} \ \ \ \mbox{ for all $N \ge 2$}
	\end{align*}
	with probability at least $1-\delta$. We have $C(\delta) \asymp \log \delta^{-1}$.
\end{cor}
Concerning the proof of Corollary~\ref{cor0} we will refer to Section~\ref{sec_proof_cor1}. \\

There is an interesting connection of Corollary~\ref{cor0} to the theory of normal numbers which is worth to be mentioned: it is well-known that a real number $\alpha$ is normal to base 2, if and only if the sequence $(\{2^{n-1} \alpha\})_{n \ge 1}$ is uniformly distributed modulo one (see \cite[Chapter~1, Theorem~8.1]{KN}). Hence the $\bsalpha$'s which satisfy the discrepancy estimate in Corollary~\ref{cor0} are $d$-tuples of normal numbers to base 2. (By another well-known result due to Borel~\cite{bor} almost all numbers $\alpha \in [0,1]$ are normal to every base $b \ge 2$.)

It should also be mentioned, that metrical bounds on the star-discrepancy of classical Kronecker-sequences for fixed $d$ have been given by Beck in \cite{beck}.\\

In this paper we study digital Kronecker-sequences which are a ``non-archimedean analog'' to classical Kronecker-sequences and which fit into the class of  digital $(t,s)$-sequences. This concept was introduced by Niederreiter~\cite[Section~4]{N92} and further investigated by Larcher and Niederreiter~\cite{LN93}. We will give a digital analog of Theorem~\ref{thm0}. In the next section we provide the necessary definitions and we formulate the metrical discrepancy estimate. The proof of our result will be presented in Section~\ref{sec_proof}.

\section{Digital Kronecker-sequences and formulation of the main result}

Let $q$ be a prime number and let $\mbb{Z}_q=\{0,\ld,q-1\}$ be the finite field of order $q$ with the usual arithmetic operations modulo $q$. We denote the field of formal Laurent series over $\mathbb{Z}_q$ in the variable $t^{-1}$ by $\Z$. Elements of $\Z$ are of the form 
\begin{align}\label{defLser}
g=\sum_{i=w }^{\infty} g_i t^{-i},
\end{align}
where $w$ is an arbitrary integer and all $g_i \in \mathbb{Z}_q$ with $g_w \not=0$. 
Note that $\Z$ contains the polynomial ring $\mbb{Z}_q[t]$ over $\mathbb{Z}_q$. 

For a formal Laurent series $g$ of the form \eqref{defLser} we define its ``fractional part'' by $$\{g\}:=\sum_{i=\max(1,w)}^{\infty} g_it^{-i}.$$ Let
\begin{align*}
\Zb:=& \left\{\{g\} \ : \ g \in \Z \right\}\\
 = & \{g \in \Z  \ : \ g \mbox{ of the form \eqref{defLser} with $w \ge 1$}\}
\end{align*}
and define further  
\begin{align*}
\phi:\Zb \rightarrow [0,1),\ \ \ \sum_{i=1}^{\infty} g_i t^{-i}\mapsto \sum_{i=1}^{\infty}g_i q^{-i}.
\end{align*}
Applied to vectors the operations $\{ \cdot \}$ and $\phi$ are understood component-wise. 

We associate a nonnegative integer $n$ with $q$-adic expansion $n=n_0+n_1 q+\cdots+n_r q^r$, where $n_0,\ldots,n_r \in \mathbb{Z}_q$ with the polynomial $n(t)=n_0+n_1 t+\cdots +n_t t^r$ in $\mbb{Z}_q[t]$ and vice versa.

\begin{defi}\rm
For a given $d$-tuple $\bsf=(f_1,\ldots,f_d)$ of elements of $\Z$ the sequence $\mathcal{S}(\bsf)=(\bsy_n)_{n \ge 0}$ given by $$\bsy_n=\phi(\{n \bsf\}) =(\phi(\{n f_1\}),\ldots,\phi(\{n f_d\}))\ \ \ \mbox{ for all $n \in \mathbb{N}_0$}$$ is called a {\it digital Kronecker-sequence} over $\mathbb{Z}_q$. Note that the multiplication of the polynomial $n$ and the Laurent series $f_j$ is carried out in $\Z$. (Obviously it suffices to choose $\bsf \in (\Zb)^d$.)
\end{defi}

In order to prove a metrical result for digital Kronecker-sequences we need to introduce a suitable probability measure on $(\overline{\ZZ}_q((t^{-1})))^d$. 

\begin{defi}\label{def_meas}\rm
By $\mu$ we denote the normalized Haar-measure on $\overline{\ZZ}_q((t^{-1}))$ and by $\mu_d$ the $d$-fold product measure on $(\overline{\ZZ}_q((t^{-1})))^d$. 
\end{defi}

\begin{rem}
The measure $\mu$ has the following rather simple shape: If we identify the elements $\sum_{k=1}^\infty g_k t^{-k}$ of $\overline{\ZZ}_q((t^{-1}))$ where $g_k\not=q-1$ for infinitely many $k$ in the natural way with the real numbers  $\sum_{k=1}^\infty g_k q^{-k}  \in [0,1)$, then, by neglecting the countably many elements where $g_k \not= q-1$ only for finitely many $k$, $\mu$ corresponds to the Lebesgue measure $\lambda$ on $[0,1)$. For example, the ``cylinder set'' $C(c_1,\ldots,c_m)$ consisting of all elements $g=\sum_{k=1}^\infty g_k t^{-k}$ from $\overline{\ZZ}_q((t^{-1}))$ with $g_k=c_k$ for $k=1,\ldots,m$ and arbitrary $g_k \in \ZZ_q$ for $k \ge m+1$ has measure $\mu( C(c_1,\ldots,c_m) )=q^{-m}$. 
\end{rem}

Metrical results for the star discrepancy of digital Kronecker-sequences for fixed dimension $d$ can be found in \cite{L95,LP14}. In the following we provide a non-archimedean version of the result of L\"obbe \cite{loeb}. 

For $\bsf \in (\Zb)^d$ let $\cP_N(\bsf)=\{\bsx_1,\ldots,\bsx_N\}$ be the point set consisting of the first $N$ elements of the infinite sequence $(\bsx_n)_{n \ge 1}$ in $[0,1)^d$ with $\bsx_n=\phi(\{t^{n-1} \bsf\})$ for $n \in \mathbb{N}$.

\begin{thm}\label{thm1}
	Let $q$ be a prime number and let $N,d\in \mbb{N}$ with $N,d\geq 2$. Then for every $\varepsilon \in (0,1)$ there is a quantity $C(q,\varepsilon)>0$ such that the star discrepancy of the point set $\cP_N(\bsf)$ satisfies 
	\begin{align*}
	D^*_N(\cP_N(\bsf)) \leq C(q,\varepsilon)\sqrt{\frac{d\log d}{N}}
	\end{align*}
	with probability at least $1-\varepsilon$. The quantity $C(q,\varepsilon)$ is of order $C(q,\varepsilon) \asymp \log \varepsilon^{-1}$.
\end{thm}

The proof of this result will be presented in the next section. It should be mentioned that with some more effort the quantity $C(q,\varepsilon)$ could be given explicitly.

Again Theorem~\ref{thm1} makes an assertion for fixed $N$, i.e. for finite point sets. From this we can again deduce a metrical result for infinite sequences:

\begin{cor}\label{cor1}
Let $q$ be a prime number and let $d\in \mbb{N}$ with $d\geq 2$. Then for every $\delta \in (0,1)$ there is a quantity $C(q,\delta)>0$ such that the star discrepancy of  $\cP_N(\bsf)$ satisfies 
	\begin{align*}
	D^*_N(\cP_N(\bsf)) \leq C(q,\delta) (\log N)\sqrt{\frac{d\log d}{N}} \ \ \ \mbox{ for all $N \ge 2$}
	\end{align*}
	with probability at least $1-\delta$. 
\end{cor}

The proof of Corollary~\ref{cor1} will be presented in Section~\ref{sec_proof_cor1}.

\section{The proof of Theorem~\ref{thm1}}\label{sec_proof}

The proof of Theorem~\ref{thm1} is inspired by the techniques used in \cite{loeb}. The difficulty here is that we are concerned with polynomial arithmetic over finite fields instead of the usual integer arithmetic. 

Throughout the proof we tacitly assume that all components of $\bsf$ belong to the class of Laurent series $\sum_{k=1}^\infty g_k t^{-k}$ of $\overline{\ZZ}_q((t^{-1}))$ where $g_k\not=q-1$ for infinitely many $k$. Let $$\mathcal{C}=\{g \in \Zb \ : \ g_k=q-1 \mbox{ for all but finitely many $k\ge 1$}\}.$$ Note that $\mathcal{C}$ is a countable set and therefore $\mu(\mathcal{C})=0$.
We denote $\overline{\mathbb{Z}}_q^*((t^{-1})):=\overline{\mathbb{Z}}_q((t^{-1}))\setminus{\mathcal{C}}$. Hence we assume that $\bsf \in (\overline{\mathbb{Z}}_q^*((t^{-1})))^d$.

\subsection{Some auxiliary results}

As in \cite{aist11,loeb} the proof will be based on Bernstein's inequality for sums of independent random variables.

\begin{lem}[\cite{bern}, Bernstein inequality]\label{lem bernstein}
	Let $N \in \mathbb{N}$ and $X_1,\ldots ,X_N$ be independent random variables with $\mbb{E}(X_i)=0$ and $|X_i| \leq C$ for $i \in \{1,\ldots,N\}$
	and some $C>0$. Then we have for any $t>0$
$$
\mbb{P}\left(\left|\sum_{i=1}^{N}{X_i} \right| > t\right) \leq 2~\exp\left(-\dfrac{t^2}{2\sum_{i=1}^N{ \mbb{E}(X_i^2)} + \frac{2Ct}{3}}\right).$$
\end{lem}

Another very important tool in our analysis are bracketing covers whose definition is recalled below. As usual, for $\bsa=(a_1,\ld,a_d)$ and $\bsb=(b_1,\ld,b_d)$ in $[0,1]^d$ we write $\bsa \leq \bsb$ if and only if $a_i \leq b_i$ for all $i \in \{1,\ld, d\}$.

\begin{defi}\rm
Let $\delta >0$. A subset $\tau \subseteq [0,1]^d \times [0,1]^d$ is called a \emph{$\delta$-bracketing cover} if for every $\bsx\in [0,1]^d$ there exists $(\bsv,\bsw) \in \tau$ such that $\bsv \leq \bsx \leq \bsw$ and $\leb([\bo{0},\bsw) \backslash [\bo{0},\bsv))\leq \delta$.
\end{defi}

The following result about the number of elements of a $\delta$-bracketing cover is due to Gnewuch:

\begin{lem}[Gnewuch {\cite[Theorem 1.15]{gnew}}]
	For any $d\in \mbb{N}$ and any $\delta>0$ there exists a $\delta$-bracketing cover $\tau$ with 
	\begin{center}
		$|\tau| \leq \frac{1}{2}(2 {\rm e})^d(\delta^{-1} + 1)^d$\,.
	\end{center}
\end{lem}

From this result L\"obbe deduced the following corollary:

\begin{cor}[L\"obbe {\cite[Corollary 2.3]{loeb}}]\label{Cor delta brack}
	Let $d,h \in \mbb{N}$ and $q \geq 2$, then there exists a $q^{-h}$-bracketing cover $\tau_h$ with 
	\begin{enumerate}
		\item $|\tau_h|\leq \frac{1}{2} (2{\rm e})^d(q^{h+2} +1)^d$, and
		\item for $(\bsv,\bsw) \in \tau_h$ and every $i \in \{1, \ld, d\}$ there exist $a_i \in \{0,1, \ld, q^{h+1+\lo}\}$ and $b_i \in \{0,1,\ld,q^{h+2+\lo}\}$ such that
			$$v_i= \frac{a_i}{q^{h+1+ \lceil \log_q d \rceil }} \ \ \ \mbox{ and } \ \ \ w_i= \frac{w_i}{q^{h+2+\lceil \log_qd \rceil }}.$$
	\end{enumerate}
\end{cor}


\subsection{Preliminaries}

Let $N,d \in \mbb{N}$ and fix some $H \in \mbb{N}$. For $h \in \{1,\ld,H\}$ let $\tau_h$ be a $q^{-h}$-bracketing cover of $[0,1)^d$ with elements described as in Corollary \ref{Cor delta brack}. Let $\bsy\in [0,1)^d$.
We are going to define inductively a finite sequence of points $\bsbeta_h(\bsy) \in [0,1)^d$ for $h \in \{0, \ld , H+1\}$ in the following way:

\begin{enumerate}

	\item Let $\bsbeta_H(\bsy), \bsbeta_{H+1}(\bsy) \in [0,1)^d$ with $\bsbeta_H(\bsy) \leq \bsy \leq \bsbeta_{H+1}(\bsy)$ and $(\bsbeta_H(\bsy), 
		\bo{\be}_{H+1}(\bsy)) \in \tau_H$.
	
	\item For $h \in \{1,\ld,H-1\}$ let $\bsbeta_h(\bsy) \in [0,1)^d$ be such that there exists a point $\tb{w} \in [0,1)^d$ with $\bsbeta_h(\bsy) \leq 
		\bsbeta_{h+1}(\bsy) \leq \tb{w}$ and $(\bsbeta_h(\bsy),\tb{w}) \in \tau_h$.
	
	\item Set $\bsbeta_0(\bsy)=\bszero=(0,\ldots,0)$, the $d$-dimensional zero-vector.
	
	\item Additionally we choose the points $\be_h$ such that the following property is fulfilled. For $\tb{x},\bsy \in [0,1)^d$ and $h \in \{0, \ld ,H-1\}$ we have that
		\begin{center}
			$\bsbeta_{h+1}(\bsy) = \be_{h+1}(\tb{x}) \Rightarrow \be_{h}(\bsy) = \be_{h}(\tb{x})$\,.
		\end{center}
	
\end{enumerate}
Note that the sequence of points $\be_h(\bsy)$ is well defined for $h\in \{0,\ld,H+1\}$ since we choose $\tau_h$ to be a $q^{-h}$-bracketing cover.
For $\bsy \in [0,1)^d$ we observe the following properties for the finite sequence $\be_h(\bsy)$:\\
We have

\begin{enumerate}
	\item $\bszero=\be_0(\bsy) \leq \be_1(\bsy) \leq \cdots \leq \be_H(\bsy) \leq \bsy \leq \be_{H+1}(\bsy) \leq \bsone$; \label{prop 1}
	
	\item for all $h \in \{0,\ld ,H-1\}$ there exists $\tb{w} \in [0,1)^d$ such that $\be_h(\bsy) \leq \be_{h+1}(\bsy) \leq \tb{w}$ and $(\be_h(\bsy),\tb{w}) 
		\in \tau_h$. \label{prop 2} Additionally we have that $(\be_H(\bsy),\be_{H+1}(\bsy)) \in \tau_H$;
				
	\item for all $h \in \{0,\ld , H\}$ and $i\in \{1, \ld ,d\}$ we have that \label{prop 3}
				\begin{center}
					$(\be_h(\bsy))_i = q^{-(h+1 + \lo)}a_{h,i}$ 
				\end{center}
				and 
				\begin{center}
					$(\be_{H+1}(\bsy))_i= q^{-(H+2 + \lo)}b_{H+1,i}$ 
				\end{center}
				for $a_{h,i} \in \left\{0,1,\ld,q^{h+1+\lo} \right\}$ and $b_{H+1,i} \in \left\{0,1,\ld,q^{H+2+\lo} \right\}$.
\end{enumerate}
The properties \ref{prop 1}. and \ref{prop 2}. are an immediate consequence of the definition of the $\be_h(\bsy)$ and property \ref{prop 3}. follows directly
by Corollary \ref{Cor delta brack}.

Moreover, for $h \in \{0,\ld, H\}$ we define $K_h(\bsy):=[\bszero,\be_{h+1}(\bsy)) \backslash [\bszero,\be_{h}(\bsy))$ and observe that the $K_h(\bsy)$ are pairwise disjoint sets and by the definition respectively property 2. of $\be_h(\bsy)$ we obtain

\begin{align}\label{union Kh}
\bigcup_{h=0}^{H-1} K_h(\bsy) \subseteq [\bszero,\bsy) \subseteq \bigcup_{h=0}^{H}K_h(\bsy) \ \ \ \text{ and }\ \ \ \lambda(K_h(\bsy)) \leq q^{-h}.
\end{align}

Finally for $\{0,\ld,H\}$ define $S_h:= \left\{K_h(\bsy) ~:~ \bsy \in [0,1)^d \right\}$. Note that by definition of the $\be_h(\bsy)$ and Corollary $\ref{Cor delta brack}$ we have

\begin{align*}
\left|S_H \right|= \left|\left\{(\be_{H}(\bsy),\be_{H+1}(\bsy)) ~:~ \bsy \in [0,1)^d \right\} \right| \leq |\tau_{H}| \leq \frac{1}{2}(2{\rm e})^d(q^{H+2} + 1)^d.
\end{align*}
With point 4. in the definition of the $\be_h(\bsy)$ we get for $h\in \{0,\ld,H-1\}$ that

\begin{align*}
\left|S_h\right| =\left| \left\{ \be_{h+1}(\bsy) \ :\ \bsy \in [0,1)^d \right\} \right| \leq \left|\tau_{h+1} \right| \leq \frac{1}{2}(2e)^d(q^{h+3} + 1)^d.
\end{align*}

Fix $\bsy \in [0,1)^d$. In order to simplify the notation from now on we will write $\be_h$ and $K_h$ instead of $\be_h(\bsy)$ and $K_h(\bsy)$, respectively . Then by \eqref{union Kh} we get that

\begin{align}\label{sum geq}
	\sum_{n=1}^N{\tb{1}_{[\bszero,\bsy)}(\tb{x}_n)} \geq \sum_{n=1}^N{\tb{1}_{[\bszero,\be_H)}(\tb{x}_n)} = \sum_{h=0}^{H-1}{\sum_{n=1}^N{\Big(\tb{1}_{K_h}(\tb{x}_n) - \leb(K_h)\Big)}} + 
	N\sum_{h=0}^{H-1}{\leb(K_h)}
\end{align}
and
\begin{align}\label{sum leq}
	\sum_{n=1}^N{\tb{1}_{[\bszero,\bsy)}(\tb{x}_n)} \leq \sum_{n=1}^N{\tb{1}_{[\bszero,\be_{H+1})}(\tb{x}_n)} = \sum_{h=0}^{H}{\sum_{n=1}^N{\Big(\tb{1}_{K_h}(\tb{x}_n) - \leb(K_h)\Big)}} + 
	N\sum_{h=0}^{H}{\leb(K_h)}.
\end{align}
Let us define the functions $\D:[0,1)^d \rightarrow [-1,1]$, $\D(\tb{x}):=\tb{1}_{K_h}(\tb{x})-\leb(K_h)$ for $h \in \{0,\ld, H\}$. A crucial step for the proof of the main result will be to use Bernstein's inequality to give a lower bound on the probability that the inequality 

\begin{align*}
	\left|\sum_{n=1}^N{\D(\tb{x}_n)}\right| \leq t_h
\end{align*}
holds simultaneously for all $h \in \{0, \ld ,H\}$ and for some $t_h>0$ to be specified later. First of all observe that $\mbb{E}(\D(\tb{x}_n))=0$, $\mbb{E}^2(\D(\tb{x}_n))=\leb(K_h)(1-\leb(K_h))$ and $\left| \D(\tb{x}_n) \right| \leq 1$ for all $h\in \{0,\ld, H\}$ and $n\in \{1,\ld ,N\}$. Unfortunately for $h \in \{0,\ld,H\}$ the random variables $\D(\tb{x}_1),\D(\tb{x}_2),\ld, \D(\tb{x}_N)$ are not independent in general. We will see how to overcome this problem in the next section.

\subsection{Independence of $\D(x_n)$}

Before we begin we point out the following easy algebraic characterization of Laurent series whose image under $\phi$ belongs to a certain type of intervals: for $p \in \overline{\mathbb{Z}}_q^*((t^{-1}))$ of the form $p=p_1 t^{-1}+p_2 t^{-2}+p_3 t^{-3}+\cdots$, for $r\in \mbb{N}$ and $k \in \{0,\ld,q^{r}-1\}$ with $q$-adic expansion $k=k_0+k_1 q+\cdots +k_{r-1} q^{r-1}$ we have that 
\begin{flalign*}
\phi(p) \in \left[\frac{k}{q^{r}},\frac{k+1}{q^{r}}\right)  \ \Leftrightarrow \  p_1=k_{r-1},\ p_2=k_{r-2},\ \ldots,\ p_{r} =k_0.
\end{flalign*}

Throughout the proof the underlying probability measure is the measure $\mu_d$ from Definition~\ref{def_meas}. However, out of habit we will in the following denote the probability by $\mathbb{P}$.

\begin{lem}\label{independence}
	Let $\kappa_h := \log_2(h+2+\lceil \log_qd \rceil)$ and let $\gamma \in \{0, \ldots , 2^{\kappa_{h}}-1\}$. Moreover, let
	\begin{align*}
		Q(N,\kappa_h,\gamma):=\left\{n \in \{1,\ldots,N\} ~:~ n \equiv \gamma \pmod{2^{\kappa_h}}\right\}.
	\end{align*}
	Then for $n_1,\ld,n_l \in Q(N,\kappa_h,\gamma)$ and $l \in \{1,\ld, |Q(N,\kappa_h,\gamma)|\}$ the random variables $\D(\tb{x}_{n_1}), \D(\tb{x}_{n_2}) ,\ld ,\D(\tb{x}_{n_l})$ are independent, i.e.
	
	\begin{align*}
		\mbb{P}(\D(\tb{x}_{n_1})=c_1 ,\ld , \D(\tb{x}_{n_l})=c_l) = \prod_{r=1}^l\mbb{P}(\D(\tb{x}_{n_r}) = c_r)\,.
	\end{align*}
\end{lem} 

\begin{proof}
        The proof is based on the ideas from \cite{loeb}.
	We will show the case $l=2$. The general case follows by induction. Let $h\in \{0,\ld,H\}, ~\gamma \in \{0,\ld,2^{\kappa_h}-1\}$ and $n,m \in 
	Q(N,\kappa_h,\gamma)$ with $n > m$. We want to show that $\Delta_{K_h}(\tb{x}_n), \Delta_{K_h}(\tb{x}_m)$ are independent.  To this end we consider the following decomposition of $[0,1)^d$:
	\begin{align*}
		\Sigma_{n-1}:=\left\{ \prod_{i=1}^{d} \left[\frac{a_i}{q^{n-1}},\frac{a_i+1}{q^{n-1}}\right) \ : \ a_i \in \{0,\ld,q^{n-1}-1\} \right\}.
	\end{align*}
	Since the underlying structure of the sequence $(\tb{x}_k)_{k\geq1}$ is $\Zb$ we are considering the preimage of $\Sigma_{n-1}$. 
	\begin{align*}
		\Lambda_{n-1}:=\left\{ \phi^{-1}(S ) \cap \left(\overline{\mathbb{Z}}_q^*((t^{-1}))\right)^d \ : \  S\in \Sigma_{n-1} \right\}, 
	\end{align*}
	\noindent

  For $A=(a_{i,j})_{i=1,j=1}^{d,n-1} \in \mbb{Z}_q^{d\times (n-1)}$ let us define
	\begin{align*}
        B_A:=\prod_{i=1}^d \left\{g \in \overline{\mathbb{Z}}_q^*((t^{-1})) \ : \ (g_1,\ld,g_{n-1})=(a_{i,1},\ld,a_{i,n-1}) \right\}.
	\end{align*}
	\noindent
	where $(a_{i,1},\ld,a_{i,n-1})$ is the $i$-th row of $A$. One can easily check that 
	\begin{align*}
	\Lambda_{n-1}=\left\{B_A \ : \ A\in \mbb{Z}_q^{d\times (n-1)}\right\}.
	\end{align*}
	
	\noindent
	For matrices $A_1,A_2 \in \mbb{Z}_q^{d \times (n-1)}$, $A_j=(a_{j,i,k})_{i=1,k=1}^{d,n-1}$ for $j \in \{1,2\}$ we define
	\begin{align*}
	\alpha_{A_1,A_2} : \  & (\Zb)^d \ra (\Zb)^d, \\
	 &(g^{(1)},\ld,g^{(d)}) \mapsto (g^{(1)}+u_{A_1A_2}^{(1)},\ld,g^{(d)}+u_{A_1A_2}^{(d)}),
	\end{align*}
	where for $i \in \{1,\ldots,d\}$, $u_{A_1 A_2}^{(i)}=\sum_{k=1}^{\infty} u_{A_1A_2,k}^{(i)}t^{-k} \in \Zb$ and 
	$$u_{A_1A_2,k}^{(i)} = 
		\begin{cases}
			a_{2,i,k}-a_{1,i,k}  & \mbox{ if } 1\leq k\leq n-1,\\
			0 & \mbox{ if } k>n-1.
		 \end{cases}$$
		 
	With this definition we have $$\alpha_{A_1,A_2}(B_{A_1})=B_{A_2}.$$

	Before we can prove the independence of $\Delta_{K_h}(\tb{x}_n)$ and $\Delta_{K_h}(\tb{x}_m)$ we need to show four claims:
	
	\begin{claim}\label{cl1}
		Let $c\in \mbb{R}$, $A_1,A_2 \in \Zm$, $\tb{f} \in B_{A_1}$ and $(\bsy_n)_{n \geq1}$ in $[0,1)^d$ with $\bsy_n=\phi(\{t^{n-1} \overline{\bsf}\})$ 
                with $\overline{\tb{f}} =\alpha_{A_1A_2}(\tb{f})$. Then we have that
		\begin{align*}
			\mbb{P}\big(\Delta_{K_h}(\tb{x}_n) = c ~|~ \tb{f} \in B_{A_1})=\mbb{P}(\Delta_{K_h}(\bsy_n)=c ~|~ \overline{\tb{f}} \in B_{A_2}\big)\,.
		\end{align*}
	\end{claim}

	\noindent
	{\it Proof of Claim~\ref{cl1}:}
	For $i \in \{1,\ld,d\}$ we have that
	$$y_{n,i}=\phi(\{t^{n-1}\overline{f}^{(i)}\}) = \phi(\{t^{n-1}f^{(i)} + t^{n-1}u_{A_1A_2}^{(i)})\}) = \phi(\{t^{n-1}f^{(i)}\})=x_{n,i}.$$
	Note that the second last equality is true because $u_{A_1A_2,k}^{(i)}=0$ for $k \geq n$.
	Additionally it holds that $\tb{f} \in B_{A_1} \Leftrightarrow \overline{\tb{f}} \in B_{A_2}$. Therefore the claim follows.\qed

	\begin{claim}\label{coefficients}
	Let $h \in \{0,\ld,H\}$ and $\tb{p}=(p^{(1)},\ld,p^{(d)}) \in (\overline{\mathbb{Z}}_q^*((t^{-1})))^d$ with $p^{(i)} = \sum_{j =1}^{\infty} p^{(i)}_j t^{-j}$. Then the  $d(h+2+\lo)$ coefficients $p^{(i)}_1, \ld ,p^{(i)}_{h+2+\lo}$ for $i \in \{1,\ldots,d\}$ determine if $\phi(\tb{p})\in K_h$.
	\end{claim}

	\noindent
	{\it Proof of Claim~\ref{coefficients}:}
	For $\tb{p}=(p^{(1)},\ld,p^{(d)}) \in (\overline{\mathbb{Z}}_q^*((t^{-1})))^d$, we have that 
	\begin{equation}\label{help}
		\begin{aligned}[b]
			\phi(\tb{p}) \in K_h  &\Leftrightarrow \phi(\tb{p}) \in [\bszero,\be_{h+1}) \backslash [\bszero,\be_h)\\ 
				&\Leftrightarrow \forall i \in \{1,\ld,d\} : \phi(p^{(i)})  \in [0,\beta^{(i)}_{h+1}) \text{ and }\\ 
				& \qquad \exists j\in \{1,\ld,d\}: \phi(p^{(j)}) \in [\beta_{h}^{(j)},1),                              
		\end{aligned}
	\end{equation}
	where $\be_{h+1}=(\beta_{h+1}^{(1)},\ldots \beta_{h+1}^{(d)})$ with $$\beta_{h+1}^{(i)}= \frac{b_i}{q^{h+2+\lo}} \ \ \mbox{ for some }\ \ b_i \in \{0,1,\ld, q^{h+2+\lo}-1\}$$ and similarly $\be_{h}=(\beta_{h}^{(1)},\ldots \beta_{h}^{(d)})$ with $$\beta_{h}^{(i)}= \frac{\bar{b}_i}{q^{h+1+\lo}} \ \ \mbox{ for some }\ \ \bar{b}_i \in \{0,1\ld, q^{h+1+\lo}-1\}.$$
        We can write $$[0,\beta^{(i)}_{h+1})=\bigcup_{k=0}^{b_i-1} \left[\frac{k}{q^{h+2+\lo}},\frac{k+1}{q^{h+2+\lo}}\right).$$ Hence $\phi(p^{(i)})  \in [0,\beta^{(i)}_{h+1})$ if and only if there exists a $k \in \{0,\ldots,b_i-1\}$ such that $$\phi(p^{(i)}) \in \left[\frac{k}{q^{h+2+\lo}},\frac{k+1}{q^{h+2+\lo}}\right).$$ Since $\phi(p^{(i)})=\sum_{j=1}^{\infty} p^{(i)}_j q^{-j}$ the last condition is satisfied if and only if $$p^{(i)}_1=k_{h+1+\lo}, p^{(i)}_2=k_{h+\lo},\ldots, p^{(i)}_{h+2+\lo} =k_0,$$ whenever $k$ has $q$-adic expansion $k=k_0+k_1 q+\cdots +k_{h+1+\lo} q^{h+1+\lo}$.

        In the same vein we can write $$[\beta^{(j)}_h,1)=\bigcup_{\ell=\bar{b}_j}^{q^{h+1+\lo}-1} \left[\frac{\ell}{q^{h+1+\lo}},\frac{\ell+1}{q^{h+1+\lo}}\right).$$ Hence $\phi(p^{(j)})  \in [\beta^{(j)}_{h},1)$ if and only if there exists a $\ell \in \{\bar{b}_j,q^{h+1+\lo}-1\}$ such that $$\phi(p^{(j)}) \in \left[\frac{\ell}{q^{h+1+\lo}},\frac{\ell+1}{q^{h+1+\lo}}\right).$$ Since $\phi(p^{(j)})=\sum_{k=1}^{\infty} p^{(j)}_k q^{-k}$ the last condition is satisfied if and only if $$p^{(j)}_1=l_{h+\lo}, p^{(j)}_2=l_{h+\lo-1},\ldots, p^{(j)}_{h+1+\lo} =l_0,$$ whenever $\ell$ has $q$-adic expansion $\ell=l_0+l_1 q+\cdots +l_{h+\lo} q^{h+\lo}$.

	Together with \eqref{help} it follows that the coefficients $p^{(i)}_1, \ld ,p^{(i)}_{h+2+\lo}$ for $i \in \{1,\ldots ,d\}$ determine whether or not 
	$\phi(\tb{p})$ belongs to $K_h$. This proves the second claim.\qed \\

	Recall that $m \in Q(N,\kappa_h,\gamma)$ and $m < n$. Define 
	$$\delta_m:\Zb \ra \Zb,~p \mapsto \{t^{m-1}p\}.$$

	\begin{claim}\label{Delta constant}
	For all $h \in \{0,\ld,H\}$ and for all $A \in \Zm$ we have that $\Delta_{K_h}$ is constant on $\phi(\delta_m(B_A))$.
	\end{claim}
	
	\noindent
	{\it Proof of Claim~\ref{Delta constant}:}
	Let $\tb{p}=(p^{(1)},\ldots,p^{(d)}) \in B_A$ with $p^{(i)}=\sum_{j=1}^{\infty}p^{(i)}_j t^{-j}$. Note that for each $i \in \{1,\ld,d\}$ the first $n-1$ coefficients $p^{(i)}_1,\ld,p^{(i)}_{n-1}$ of $p^{(i)}$ are equal to the entries in the $i$-th row of $A$. Now we have 
	\begin{align*}
	\delta_m\left(\sum_{j=1}^{\infty}p^{(i)}_jt^{-j}\right)=\left\{\sum_{j=1}^{\infty}p^{(i)}_j t^{m-1-j}\right\}=\sum_{j=1}^{\infty}p^{(i)}_{m-1+j}t^{-j}.
	\end{align*}
        Because of Claim~\ref{coefficients}, the coefficients $p^{(i)}_{m},\ld,p^{(i)}_{m+h+1+\lo}$ for $i \in 
	\{1,\ld,d\}$ determine if $\phi(\delta_m(\tb{p})) \in K_h$. Since $n-m \geq h+2+ \lo$ these coefficients are fixed by the choice of $B_A$. Hence it follows that $\phi(\delta_m(B_A)) \cap K_h \in
	\{\emptyset, \phi(\delta_m(B_A))\}$. Therefore the function $\Delta_{K_h}(\boldsymbol{x})=\tb{1}_{K_h}(\boldsymbol{x})-\lambda(K_h)$ is constant on $\phi(\delta_m(B_A))$. This proves the claim. \qed
	\\

	Define for $c \in \mbb{R}$,
	\begin{align*}
		\Lambda_{K_h,c}:=\left\{B_A \in \Lambda_{n-1} \ : \ \Delta_{K_h}(\phi(\delta_m(B_A)))=c \right\}.
	\end{align*}
	Note that $\Lambda_{K_h,c}$ is well-defined according to Claim \ref{Delta constant}.
	
	\begin{claim}\label{final claim}
		Let $c \in \mbb{R}$ and $h \in \{0,\ld,H\}$. Then we have
		\begin{align*}
			\Delta_{K_h}(\tb{x}_m)=c \Leftrightarrow \exists B_A \in \Lambda_{K_h,c} \ \mbox{ such that }\  \tb{f} \in B_A.
		\end{align*}
	\end{claim}
	
	\noindent
	{\it Proof of Claim~\ref{final claim}:}
	Let $c\in \mbb{R}$ and suppose that there exists $B_A \in \Lambda_{K_h,c}$ such that $\tb{f} \in B_A$. Since $\tb{x}_m=\phi(\delta_m(\tb{f}))$ we have
	$$\Delta_{K_h}(\tb{x}_m) =\Delta_{K_h}(\phi(\delta_m(\tb{f}))).$$ Since $\delta_m(\tb{f}) \in \delta_m(B_A)$ we get that $\Delta_{K_h}(\tb{x}_m)=c$.
	
        Now assume that $\Delta_{K_h}(\tb{x}_m)=c$ which is equivalent to $\Delta_{K_h}(\phi(\delta_m(\tb{f})))=c$. Now there exists $A'\in \Zm$ such that $B_{A'} \in
	\Lambda_{K_h,c}$ and	$\delta_m(\tb{f}) \in \delta_m(B_{A'})$ and we get that	$(f^{(i)}_{m+1},\ld,f^{(i)}_{n-1})=(a'_{i,m+1},\ld,a'_{i,n-1})$ for $i\in\{1,\ld,d\}$.\\
	On the other hand 
        there exists $A \in \Zm$ with $\tb{f}	\in	B_{A}$. We obtain 
	$(f^{(i)}_{m+1},\ld,f^{(i)}_{n-1})=(a_{i,m+1},\ld,a_{i,n-1})$ for $i \in \{1,\ld ,d\}$. Altogether
	we have that $$(a'_{i,m+1},\ld,a'_{i,n-1})=(a_{i,m+1},\ld,a_{i,n-1})\ \ \mbox{ for }\ \ i \in \{1,\ld,d\}.$$ Now it follows by Claim~\ref{coefficients} that
	\begin{align*}
	\Delta_{K_h}(\phi(\delta_m(B_{A})))=\Delta_{K_h}(\phi(\delta_m(B_{A'})))=c.
	\end{align*}
	This means that $B_{A} \in \Lambda_{K_h,c}$ and $\bsf\in B_{A}$ and this proves the claim.\qed
	\\

	Now we can prove the independence of $\Delta_{K_h}(\tb{x}_n)$ and $\Delta_{K_h}(\tb{x}_m)$ for $n,m \in 
	Q(N,\kappa_h,\gamma)$ with $n > m$. For $c\in \mbb{R}$ and $B_{A'} \in \Lambda_{n-1}$ we have
	\begin{align}\label{indep 1}
	\mbb{P}(\Delta_{K_h}(\tb{x}_n)=c) &=\sum_{B_A \in \Lambda_{n-1}}{\mbb{P}(\Delta_{K_h}(\tb{x}_n) = c ~|~ \tb{f} \in B_A)\mbb{P}(\tb{f} \in B_A)} \nonumber\\
		&=\mbb{P}(\Delta_{K_h}(\bsy_n)=c ~|~ \overline{\tb{f}} \in B_{A'}) \sum_{B_A \in \Lambda_{n-1}}{\mbb{P}(\tb{f} \in B_A)} \nonumber\\
		&=\mbb{P}(\Delta_{K_h}(\bsy_n)=c ~|~ \overline{\tb{f}} \in B_{A'}),
	\end{align}
	where we used Claim~\ref{cl1} with $A_1=A$ and $A_2=A'$. By Claim \ref{final claim} and \eqref{indep 1} we get for $c_1,c_2 \in \mbb{R}$ and $A^{\prime} \in \Zm$ that
        \begin{eqnarray*}
        \lefteqn{\mbb{P}(\Delta_{K_h}(\tb{x}_n)=c_2 ~|~ \Delta_{K_h}(\tb{x}_m)=c_1) = \frac{\mbb{P}(\Delta_{K_h}(\tb{x}_n)=c_2, \Delta_{K_h}(\tb{x}_m)=c_1)}{\mbb{P}(
	\Delta_{K_h}(\tb{x}_m)=c_1)}}\\
        	&\qquad =& \frac{\sum_{B_A \in \Lambda_{n-1}}{\mbb{P}(\Delta_{K_h}(\tb{x}_n)=c_2, \Delta_{K_h}(\tb{x}_m)=c_1 ~|~ \tb{f} \in B_A)\mbb{P}(\tb{f} \in B_A)}}
			{\mbb{P}(\Delta_{K_h}(\tb{x}_m)=c_1)}\\
		&\qquad =& \sum_{B{_A} \in \Lambda_{K_h,c_1}}{\mbb{P}(\Delta_{K_h}(\tb{x}_n)=c_2 ~|~ \tb{f} \in B_A)}\frac{\mbb{P}(\tb{f} \in B_A)}{\mbb{P}(\Delta_{K_h}(
			\tb{x}_m)=c_1)}\\
		&\qquad =& \mbb{P}(\Delta_{K_h}(\bsy_n) = c_2 ~|~ \overline{\tb{f}} \in B_{A^{\prime}}) \dfrac{\sum_{B_A \in \Lambda_{K_h,c_1}}\mbb{P}(\tb{f} \in B_A)}
			{\mbb{P}(\Delta_{K_h}(\tb{x}_m)=c_1)}\\
 		&\qquad =& \mbb{P}(\Delta_{K_h}(\bsy_n)=c_2 ~|~ \overline{\tb{f}} \in B_{A^{\prime}})\\
		&\qquad =& \mbb{P}(\Delta_{K_h}(\tb{x}_n) = c_2).
	\end{eqnarray*}
This implies the desired result.
\end{proof}

\subsection{Applying Bernstein's inequality and finalizing the proof of Theorem~\ref{thm1}}

We may assume that $N \ge d \log_q d$ since otherwise the discrepancy bound is trivial. First of all we set
\begin{align}\label{def H}
	H=\lc \frac{1}{2}\log_q\left(\frac{N}{d\log_q d}\right) \rc \in \mbb{N}.
\end{align}
With this choice we obtain 
\begin{align*}
	\dfrac{1}{q^H} \leq \sqrt{\dfrac{d\log_q d}{N}}\ \ \ \mbox{ and }\ \ \ q^{2H} \leq  q^2\frac{N}{d \log_qd}.
\end{align*}

Recall the definition of $Q(N,\kappa_h,\gamma)= \left\{n \in \{1,\ldots,N\} ~:~ n \equiv \gamma \mod 2^{\kappa_h} \right\}$ and note that $Q(N,\kappa_h,\gamma)$ for $\gamma \in \left\{0, \ld, 2^{\kappa_h}-1 \right\}$ are a partition of $\{1,\ld ,N\}$ and that $$\left|Q(N,\kappa_h,\gamma)\right| \leq \bigg \lfloor \dfrac{N}{2^{\kappa_h}} \bigg \rfloor + \xi\ \ \ \mbox{ for some $\xi \in \{0,1\}$.}$$ With the help of Lemma~\ref{independence} we are able to apply Bernstein's inequality (Lemma~\ref{lem bernstein}). For $h \in \{0,\ld, H\}$ we get that

\begin{align}
	\mbb{P}\left( \left| \sum_{n=1}^N{\D(\tb{x}_n)} \right| > t_h \right) &\leq \sum_{\gamma=0}^{2^{\kappa_{h}}-1}\mbb{P}\left( \left| \sum_{n \in Q(N,\kappa_h,\gamma)}
		\D(\tb{x}_n) \right| > \frac{t_h}{2^{\kappa_h}}\right) \nonumber\\
		&\leq 2 \sum_{\gamma=0}^{2^{\kappa_h}-1} \exp\left( - \frac{t_h^2/2^{2\kappa_h}}{2|Q(N,\kappa_h,\gamma)| \leb(K_h)(1-\leb(K_h)) + 2t_h/(3\cdot 2^{\kappa_h})}
			\right) \nonumber \\
		&\leq 2^{\kappa_h +1} \exp\left(- \frac{t_h^2 / 2^{\kappa_h}}{2(1+2^{\kappa_h}/N)N q^{-h} + 2t_h / 3}\right) \nonumber.
\end{align}
Since 
\begin{eqnarray*}
\frac{2^{\kappa_h}}{N} & = & \frac{h+2+ \lceil \log_q d\rceil}{N} \le  \frac{1}{N} \left(\frac{1}{2}\log_q\left(\frac{N}{d\log_q d}\right) +4+\log_q d\right)\\
& \le & \frac{1}{2} \frac{\log_q N}{N} + 4 + \frac{1}{d} \le 5,
\end{eqnarray*}
we obtain 
\begin{equation}\label{using bernstein}
 \mbb{P}\left( \left| \sum_{n=1}^N{\D(\tb{x}_n)} \right| > t_h \right) \le 2^{\kappa_h +1} \exp\left(- \frac{t_h^2 / 2^{\kappa_h}}{ 12 Nq^{-h} + \frac{2}{3}t_h} \right)\,.
\end{equation}

For the choice of $t_h$ we will distinguish two cases

\begin{align} \label{def of t}
	t_h:=	\begin{cases}
			C_1 \sqrt{Nd h q^{-h}2^{\kappa_h} } & \text{ if } h \in \{1,\ld,H\}\\
			C_2 \sqrt{N d 2^{\kappa_0} }  & \text{ if } h=0
			\end{cases}
\end{align}
for constants $C_1,C_2>0$ to be specified later.

Let us consider first the case $h\in\{1,\ld,H\}$. By $\kappa_h= \log_2( h +2 +\lo) $ we get that
\begin{align*}
2^{\kappa_h}q^h h \leq & 2^{\kappa_H}q^HH \leq H^2 q^{H} (4+\log_q d) \leq 2 q^{2 H} (4+\log_q d)\\
\le & 2 q^2 \frac{N}{d} \left(1+\frac{4}{\log_q d}\right) \le c(q) \frac{N}{d},
 \end{align*}
where $c(q)=2 q^2\left(1+\frac{4}{\log_q 2}\right)$. Thus we obtain for $h \in \{1,\ld,H\}$
\begin{align*}
t_h=C_1 \sqrt{Nd h q^{-h}2^{\kappa_h}}  \leq C_1 \sqrt{c(q)} q^{-h}N.
\end{align*}
Furthermore we get
\begin{align}\label{bound t}
\frac{t_h^2/2^{\kappa_h}}{12 q^{-h}N + \frac{2}{3}t_h} \geq \frac{C_1^2 Nd hq^{-h}}{12 q^{-h}N + \frac{2}{3}C_1 \sqrt{c(q)} q^{-h}N} = \dfrac{C_1^2dh}{ 12+ \frac{2}{3} C_1\sqrt{c(q)}}.
\end{align}

Combining \eqref{using bernstein} and \eqref{bound t} we get
\begin{align}\label{prob bound h not zero}
\mbb{P}\left(\left| \sum_{n=1}^N \D(\tb{x}_n) \right| > C_1\sqrt{N h q^{-h}2^{\kappa_h} d}\right) 
 \leq 2\exp\left(\kappa_h \log 2 - \frac{C_1^2}{12+ \frac{2}{3} C_1 \sqrt{c(q)}} \ dh \right).
\end{align}

Consider the case $h=0$, i.e. $t_0 = C_2 \sqrt{N 2^{\kappa_0} d}$. We have
\begin{align*}
2^{\kappa_0} d \leq &  (3 + \log_q d) d  \le N c(q).
\end{align*}
After continuing with the same steps as in the first case we end up with
\begin{align}\label{prob bound h equal zero}
\mbb{P}\left( \left| \sum_{n=1}^N \Delta_{K_0}(\tb{x}_n) \right| > C_2\sqrt{N 2^{\kappa_0} d} \right) \leq 2\exp\left(\kappa_0 \log 2 - \frac{C_2^2}{12 + \frac{2}{3}C_2 \sqrt{c(q)}}\ d\right).
\end{align}

Recall that $\be_h$ and $K_h$ are dependent on a point $\bsy \in [0,1)^d$, respectively. Moreover, we defined $S_h=\left\{K_h(\bsy) \ : \ \bsy \in [0,1)^d \right\}$ with $\left|S_h\right|\leq \frac{1}{2}(2 {\rm e})^d(q^{h+3} + 1)^d.$ Additionally we define 

\begin{align*}
A_{K_h,N,d}:=\left\{\bsf \in (\overline{\mathbb{Z}}_q^*((t^{-1})))^d \ : \ \left| \sum_{n=1}^N{\D(\tb{x}_n)} \right| > t_h \right\}
\end{align*}
\noindent
with $t_h$ defined as in \eqref{def of t} and set

\begin{align}\label{def C3,C4}
 C_3:=\dfrac{C_1^2}{12+ \frac{2}{3} C_1\sqrt{c(q)}}, \ \ \ \mbox{ and }\ \ \  C_4:=\frac{C_2^2}{12+\frac{2}{3}C_2 \sqrt{c(q)}}. 
\end{align}
Then with \eqref{prob bound h not zero} and \eqref{prob bound h equal zero}  we have

\begin{eqnarray}\label{lberndiscbd}
\lefteqn{	\mbb{P} \left( \bigcap_{h=0}^H{\bigcap_{K_h \in S_h}{\left\{ \left|\sum_{n=1}^N \D(\tb{x}_n)\right| \leq  t_h \right\} }}\right) =1-\mbb{P}\left(\bigcup_{h=0}^H{\bigcup_{K_h \in S_h}{A_{K_h,N,d}}} \right) }\nonumber\\
		&\qquad \geq& 1 - \sum_{K_0\in S_0}\mbb{P}(A_{K_0,N,d}) - \sum_{h=1}^H{\sum_{K_h \in S_h}{\mbb{P}(A_{K_h,N,d})}}\nonumber \\
		&\qquad \geq& 1 - |S_0|2 {\rm e}^{\kappa_0\log 2-C_4 d}  - \sum_{h=1}^H |S_h|2 {\rm e}^{\kappa_h\log 2-C_3dh} \nonumber\\
		&\qquad \geq& 1- (2q^3+2)^d {\rm e}^{d(1-C_4) + \kappa_0\log 2} - \sum_{h=1}^H (2 q^{h+3}+2)^d {\rm e}^{d(1-C_3h) + \kappa_h\log 2}.
\end{eqnarray}

We will now choose $C_1=C_1(\varepsilon)$ and $C_2= C_2(\varepsilon)$ such that 
\begin{align} \label{first ineq}
	(2q^3+2)^d {\rm e}^{d(1-C_4) + \kappa_0\log 2} \leq \frac{\varepsilon}{2}
\end{align}
\noindent
and
\begin{align}\label{second ineq}
	(2 q^{h+3}+2)^d {\rm e}^{d(1-C_3h) + \kappa_h\log 2} \leq \frac{\varepsilon}{2^{h+1}}.
\end{align}
\noindent
From \eqref{lberndiscbd}, \eqref{first ineq} and \eqref{second ineq} we then obtain that
\begin{align*}
\mbb{P} \left( \bigcap_{h=0}^H{\bigcap_{K_h \in S_h}{\left\{ \left|\sum_{n=1}^N \D(\tb{x}_n)\right| \leq  t_h \right\} }}\right) \geq 1-\varepsilon.
\end{align*}

Inequality \eqref{first ineq} is equivalent to $$C_4 \ge \frac{1}{d}\left(d \log(2 q^3+2) + d+ \log(2+ \lceil \log_q d\rceil) + \log \frac{2}{\varepsilon}\right).$$ This is certainly satisfied for the choice $$C_4=\log(2 q^3+2)+2+\log \frac{2}{\varepsilon}=\log\left(\frac{4 (q^3+1) {\rm e^2}}{\varepsilon}\right).$$ With \eqref{def C3,C4} it follows that we choose $$C_2=C_4 \frac{1}{3} \sqrt{c(q)}+\sqrt{C_4^2 \frac{1}{9} c(q)+12 C_4} \asymp \log \frac{1}{\varepsilon}.$$

Inequality \eqref{second ineq} is equivalent to $$C_3 \ge \frac{1}{d h}\left(d \log(2 q^{h+3}+2) + d+ \log(h+2+ \lceil \log_q d\rceil) + \log 2^h+ \log \frac{2}{\varepsilon}\right).$$ This is certainly satisfied for the choice $$C_3=\log(2 (q^4+1)) +2+ \frac{\log 2}{2}+\log \frac{2}{\varepsilon}=\log\left(\frac{4\sqrt{2} (q^4+1) {\rm e}^2}{\varepsilon} \right).$$
With \eqref{def C3,C4} it follows that we choose $$C_1=C_3 \frac{1}{3} \sqrt{c(q)}+\sqrt{C_3^2 \frac{1}{9} c(q)+12 C_3} \asymp \log \frac{1}{\varepsilon}.$$ 

Finally by \eqref{prob bound h equal zero}, \eqref{prob bound h not zero}, \eqref{sum leq} we obtain with probability at least $1-\varepsilon$ 

\begin{eqnarray}\label{dependence on theta}
\lefteqn{\sum_{n=1}^N \tb{1}_{[\bszero,\bsy)}(\tb{x}_n) \leq \sum_{n=1}^N\sum_{h=0}^H{\D(\tb{x}_n)} + N\leb([\bszero,\be_{H+1}))}\nonumber \\
&\leq & \sum_{n=1}^N \Delta_{K_0}(\tb{x}_n) + \sum_{h=1}^H\sum_{n=1}^N \D(\tb{x}_n) + N \left(\leb([\bszero,\bsy))+\leb([\bszero,\be_{H+1}))-\leb([\bszero,\bsy))\right) \nonumber\\
&\leq & \sqrt{Nd} \left(C_2\sqrt{2^{\kappa_0}} + \sum_{h=1}^H C_1\sqrt{2^{\kappa_h}}\sqrt{hq^{-h}} \right) + N(\leb([\bszero,\bsy))+q^{-H})\nonumber \\
&\leq & \sqrt{Nd}\left( C_2\sqrt{2^{\kappa_0}} + \sum_{h=1}^{\infty} C_1\sqrt{2^{\kappa_h}}\sqrt{hq^{-h}} \right)  + \sqrt{Nd\log_q d} + N\leb([\bszero,\bsy))\nonumber\\
&\leq & \sqrt{Nd} \left( C_2\sqrt{2^{\kappa_0}} + \sum_{h=1}^{\infty}{ C_1\sqrt{(h +3)hq^{-h}} } + \sqrt{\log_qd}\sum_{h=1}^{\infty}{ C_1\sqrt{hq^{-h}} }\right)\nonumber\\ 
& &\qquad + \sqrt{Nd\log_q d} + N\leb([\bszero,\bsy),
\end{eqnarray}
where we used that $\leb \left([\bszero,\be_{H+1})\right) - \leb\left([\bszero,\bsy)\right)\leq \leb(K_H(y)) \leq q^{-H}$. 

By the choices for $C_1$ and $C_2$ we obtain that
\begin{align}\label{first impl}
\frac{1}{N} \sum_{n=1}^N \tb{1}_{[\bszero,\bsy)}(\bsx_n) - \leb([\bszero,\bsy)) & \le C_5(q,\varepsilon) \sqrt{\frac{d\log d}{N}}
\end{align}
where $C_5(q,\varepsilon) \asymp \log \frac{1}{\varepsilon}$.

If we use \eqref{sum geq} instead of \eqref{sum leq} and the fact that $\leb\left([\bszero,\bsy)\right)-\leb([\bszero,\be_H)) \leq \leb(K_H(y)) \leq q^{-H}$ we get that

\begin{align}\label{second impl}
	\dfrac{1}{N}\sum_{n=1}^N \tb{1}_{[\bszero,\bsy)}(\bsx_n) - \leb([\bszero,\bsy)) \geq - C_5(q,\varepsilon) \sqrt{\frac{d\log d}{N}} 
\end{align}
with $C_5(q,\varepsilon)$ as before.

Finally \eqref{first impl} and \eqref{second impl} imply
\begin{align*}
\left|\dfrac{1}{N}\sum_{n=1}^N \tb{1}_{[\bszero,\bsy)}(\bsx_n) - \leb([\bszero,\bsy)) \right| \leq C_5(q,\varepsilon) \sqrt{\frac{d\log d}{N}}.
\end{align*}
Since $\bsy \in[0,1]^d$ was arbitrary we get that
\begin{align*}
D^*_{N}(\cP_N(\bsf)) \leq C_5(q,\varepsilon) \sqrt{\frac{d\log d}{N}}.
\end{align*}
holds with probability at least $1-\varepsilon$. This finishes the proof. $\qed$

\subsection{The proof of Corollaries~\ref{cor0} and \ref{cor1}}\label{sec_proof_cor1}

Since the proofs of the two corollaries are very similar we only present the proof of Corollary~\ref{cor1}.\\

Let $c(q)>0$ be such that $C(q,\varepsilon)$ from Theorem~\ref{thm1} satisfies $C(q,\varepsilon) \le c(q) \log \varepsilon^{-1}$.
For $\delta\in (0,1)$ and $N \ge 2$ let $\varepsilon_N=6 \delta/(\pi N)^2$ and $$A_{N}:=\left\{ \bsf \in (\Zb)^d \ : \ \ D_N^*(\cP_N(\bsf)) \le c(q) \log \varepsilon_N^{-1} \sqrt{\frac{d \log d}{N}}\right\}.$$
According to Theorem~\ref{thm1} we have $\mathbb{P}(A_N) \ge 1-\varepsilon_N$.

Set $$A:=\left\{ \bsf \in (\Zb)^d \ : \ \ D_N^*(\cP_N(\bsf)) \le c(q) \log \varepsilon_N^{-1} \sqrt{\frac{d \log d}{N}} \ \mbox{ for all $N \ge 2$}\right\}.$$ Then obviously $A=\bigcap_{N \ge 2} A_N$ and hence
\begin{eqnarray*}
\mathbb{P}(A^c) = \mathbb{P}\left( \bigcup_{N \ge 2} A_N^c\right) \le \sum_{N \ge 2} \mathbb{P}(A_N^c) \le \sum_{N \ge 2}\varepsilon_N \le \delta,
\end{eqnarray*}
where $A^c$ is the complement of $A$ in $(\Zb)^d$ and similarly for $A_N^c$. Hence $\mathbb{P}(A) \ge 1-\delta$ and the result follows. $\qed$

\begin{small}
\noindent\textbf{Authors' address:}
\\
\noindent Mario Neum\"uller and Friedrich Pillichshammer,
\\
Department of Financial Mathematics and Applied Number Theory,
Johannes Kepler University Linz, Altenbergerstr.~69, 4040 Linz, Austria\\
 \\
\noindent \textbf{E-mail:} \\
\texttt{mario.neumueller@jku.at}\\
\texttt{friedrich.pillichshammer@jku.at} 
\end{small}

\end{document}